\documentclass[11pt,letterpaper]{amsart}
\usepackage{
  amssymb, xcolor, mathtools, mleftright, xurl}
\usepackage[shortlabels]{enumitem}
\usepackage{hyperref}

\usepackage[margin=1in]{geometry}

\newtheorem{theorem}{Theorem}[section]
\newtheorem{lemma}[theorem]{Lemma}

\newtheorem{prop}[theorem]{Proposition}

\newtheorem{conj}[theorem]{Conjecture}

\theoremstyle{definition}

\newtheorem{defn}[theorem]{Definition}

\theoremstyle{remark}
\newtheorem{rem}[theorem]{Remark}

\DeclareMathOperator{\ct}{ct}

\DeclareMathOperator{\disc}{disc}
\DeclareMathOperator{\Ht}{Ht}
\DeclareMathOperator{\Htp}{Htp}
\DeclareMathOperator{\ind}{ind}
\DeclareMathOperator{\sgn}{sgn}
\newcommand{\monic}{\mathrm{monic}}
\newcommand{\prim}{\mathrm{prim}}
\DeclareMathOperator{\Res}{Res}

\DeclareMathOperator{\Disc}{Disc}

\DeclareMathOperator{\Cl}{Cl}

\newcommand{\FF}{\mathbb{F}}
\newcommand{\QQ}{\mathbb{Q}}

\newcommand{\ZZ}{\mathbb{Z}}

\renewcommand{\aa}{\mathfrak{a}}
\newcommand{\bb}{\mathfrak{b}}

\newcommand{\C}{\mathcal{C}}
\newcommand{\D}{\mathcal{D}}
\newcommand{\E}{\mathcal{E}}

\newcommand{\OO}{\mathcal{O}}

\newcommand{\cross}{\times}

\newcommand{\textand}{\quad \text{and} \quad}
\renewcommand{\to}{\mathop{\rightarrow}\limits}
\newcommand{\size}[1]{\lvert #1 \rvert}
\newcommand{\Size}[1]{\left\lvert #1 \right\rvert}
\newcommand{\floor}[1]{\left\lfloor #1 \right\rfloor}
\newcommand{\ceil}[1]{\left\lceil #1 \right\rceil}
\newcommand{\1}{\mathbf{1}}

\newcommand{\intsec}{\cap}

\newcommand{\<}{\left\langle}
\renewcommand{\>}{\right\rangle}

\renewcommand{\epsilon}{\varepsilon}

\mleftright

\title[Galois groups of twisted reciprocal polynomials]{Galois groups of reciprocal polynomials II:\\Twisted reciprocal polynomials}
\author{Theresa C. Anderson and Evan M. O'Dorney}

\begin{document}

\begin{abstract}
We study the Galois group $G_f$ of a random polynomial $f$ of height at most $H$ in the family of polynomials of degree $2n$ satisfying the twisted reciprocal relation $f(x) = x^{2n}/b^n \cdot f(b/x)$, which arise in a wide variety of applications.  Our main result is a theorem of van der Waerden--Bhargava type: the probability that $G_f$ is not the full hyperoctahedral group $S_2 \wr S_n$ is $\Theta(H^{-1}\log H)$, independent of $b$, with the leading-order group $G_1$ being of index $2$. This paper is a companion to a recent paper by the authors and Bertelli addressing reciprocal polynomials (i.e.\ the case $b = 1$).
\end{abstract}


\maketitle

\section{Introduction}

\subsection{\texorpdfstring{$b$}{b}-Reciprocal polynomials}
Let $n \geq 1$ and $b \neq 0$ be integers. We say that a polynomial $f(x) = \sum_{i=0}^{2n} a_i x^i \in \ZZ[x]$ of degree $2n$ is \emph{$b$-reciprocal} if it satisfies the linear relations
\begin{equation}
  a_{n+i} = b^i a_{n-i}, \quad i = 0, 1,\ldots, n,
\end{equation}
or, in other words, the identity
\[
  f(x) = \frac{x^{2n}}{b^n} f\left(\frac{b}{x}\right).
\]
If $\disc f \neq 0$, then we can consider the Galois group $G_f \subseteq S_{2n}$ of the roots of $f$. Note that if $\alpha \neq 0$ is a root of $f$, so is $b/\alpha$, implying that $G_f$ is a subgroup of the hyperoctahedral group $G_0 = S_2 \wr S_n$. We are interested in the question of how often $G_f \subsetneq G_0$, for fixed $n$ and coefficients bounded by $H$.

When $b = 1$, these $f$ are \emph{reciprocal polynomials}, which naturally arise as characteristic polynomials of symplectic matrices, and their Galois groups were studied in \cite{ABORecipPolys}.  Here we ask the corresponding question for the whole $b$-reciprocal family.  Of particular note is the case when $b=-1$: these $f$ are \emph{skew-reciprocal polynomials}.  Skew-reciprocal polynomials arise as characteristic polynomials of antisymplectic matrices, and have numerous applications \cite{Hua}.  Our result in fact is a uniform answer for all $b$: 
\begin{theorem} \label{thm:main}
  Let $n \geq 1$ and $b \neq 0$ be integers. Let $\E_{n,b}(H)$ be the number of separable $b$-reciprocal polynomials $f$ of degree $2n$ with coefficients in $[-H,H]$ whose Galois group is not $S_2 \wr S_n$. Then, for fixed $n$ and $b$,
  \[
    \E_{n,b}(H) \asymp H^{n} \log H.
  \]
\end{theorem}

In fact, more can be said: $100\%$ of the contribution to $\E_{n,b}(H)$ as $H \to \infty$ comes from a single Galois group, denoted $G_1$ below; all other Galois groups appear at most $O(H^n)$ times. This can be seen from our bounds and the group theory of \cite[Remark 3.6]{ABORecipPolys}, which shows that all other subgroups of $G_0 = S_2 \wr S_n$ are contained in $G_2$, $G_3$, or $S_2 \wr H$ for some $H \subsetneq S_n$. This situation contrasts favorably with the state of the art for general (nonreciprocal) degree-$n$ polynomials, where some groups such as $A_n$ have only been shown (by Bhargava \cite{Bhargava_vdW}) to appear with frequency $O(H^n)$, potentially tying the dominant group $S_{n-1}$, although $A_n$ is expected to appear only $H^{(n+2)/2 + o(1)}$ times: see \cite{MR4768827} for more details.

Additionally, our result reveals an interesting connection between the sets
\[
\{f:\deg{f} = 2n, f(x) = g(x)h(x), \deg{g} = \deg{h} = n\}
\]
and 
\[
\{f: \deg{f} = 2n, f \text{ is $b$-reciprocal, } G_f \neq S_2\wr S_n\};
\]
indeed van der Waerden showed  \cite{vdW1936} that for coefficients in $[-H,H]$, polynomials in the first set occur with frequency $H^{-n}\log{H}$, which matches our count for the second set.

We get analogous results under a variant setup that restricts $f$ to be monic:  
\begin{theorem} \label{thm:monic}
  Let $n \geq 2$ and $b \neq 0$ be integers. Let $\E_{n,b}^\monic(H)$ be the number of separable monic $b$-reciprocal polynomials $f$ of degree $2n$ with coefficients in $[-H,H]$ whose Galois group is not the full $G_0 = S_2 \wr S_n$. Then:
  \[
    \E_{n,b}^\monic(H) \asymp H^{n - 1} \log H.
  \]
\end{theorem}

These polynomials are important to study due to their appearance as characteristic polynomials. As already mentioned, any matrix from the symplectic group $\mathrm{Sp}({2n}, \Bbbk)$ has reciprocal characteristic polynomial. Similarly, any matrix from the symplectic similitude group $\mathrm{GSp}(2n, \mathbb{\Bbbk})$ has $b$-reciprocal characteristic polynomial, where $b$ is the factor by which it scales the fixed symplectic form. It is desirable to know how often a $b$-reciprocal polynomial appears as a characteristic polynomial in this group. Over finite fields, this was done by Chavdarov \cite{Chavdarov}; this paper sheds possible light on the corresponding question for symplectic similitudes over $\ZZ$. 

The majority of this paper is devoted to proving Theorem \ref{thm:main}.  We address any alterations to get Theorem \ref{thm:monic} along the way, and close by mentioning some partial progress and suggested areas of exploration for other special families of polynomials.

\subsection{Notation}
Throughout this paper, $n$ and $b$ are fixed, while $H \to \infty$. We carry over the same notational conventions as in \cite{ABORecipPolys}, as far as practical.

\subsection{Acknowledgments}
This research was supported by the NSF.

\section{\texorpdfstring{$b$}{b}-Reciprocal polynomials}
If $f$ is a $b$-reciprocal polynomial of degree $2n$, it is easy to see that there is a unique degree $n$ integer-coefficient polynomial
\[
  g(u) = c_n u^n + \cdots + c_1 u + c_0
\]
such that
\[
  f(x) = x^n g\left(x + \frac{b}{x}\right).
\]
The passage between $f$ and $g$ is bijective and linear, and it increases or decreases heights by at most a bounded factor, depending only on $n$ and $b$. Hence it is immaterial whether we count $f$ or $g$ of height at most $H$. For most purposes, it is more convenient to count $g$.

We denote the roots of $f$ by
\[
  \alpha_1, \frac{b}{\alpha_1}, \alpha_2, \frac{b}{\alpha_2}, \ldots, \alpha_n, \frac{b}{\alpha_n}.
\]
\[
  \beta_i = \alpha_i + \frac{b}{\alpha_i}.
\]
We will sometimes write $\alpha = \alpha_1$ and $\beta = \beta_1$ when the choice of root is irrelevant.

\begin{lemma} \label{lem:disc_f}
  $\disc f = b^{n(n-1)} g\bigl(2\sqrt{b}\bigr) g\bigl(-2\sqrt{b}\bigr) (\disc g)^2$.
\end{lemma}
\begin{proof}
  Express both sides in terms of the $\alpha_i$ as in \cite[Lemma 2.1]{ABORecipPolys}.
\end{proof} 

Assuming $n \geq 2$, we can establish the lower bounds $\E_{n,b}(H) \gg H^n$, $\E_{n,b}^\monic(H) \gg H^{n-1}$ by multiplying a fixed $b$-reciprocal quadratic polynomial (e.g.\ $x^2 - b$) by the $b$-reciprocal polynomials of degree $2n - 2$ and height $\leq c H$ for an appropriate $c$. We can then assume any statement that occurs for all but $O(H^{n})$ of the $O(H^{n + 1})$ polynomials $g$ of height at most $H$, and all of $O(H^{n-1})$ of those that are monic. For example, we may assume that $g$ is irreducible and that $g(\pm 2\sqrt{D})$ are nonzero. We have the tower of number fields
  \[
    K_f = \QQ(\alpha) \quad \supseteq \quad K_g = \QQ(\beta) \quad \supset \quad \QQ,
  \]
  where $K_g$ is an extension of degree $n$, while $K_f/K_g$ is of degree at most $2$, given by $K_f = K_g(\sqrt{\beta^2 - 4 b})$. Let $\widetilde{K}_g$ and $\widetilde{K}_f$, respectively, be the splitting fields of $g$ and $f$, and let $G_g$ and $G_f$ be their respective Galois groups, which are subgroups of $S_n$ and $G_0$, with $G_f \twoheadrightarrow G_g$ under the natural projection $G_0 \twoheadrightarrow S_n$. By the main result of Bhargava \cite[Theorem 1]{Bhargava_vdW}, we can assume that $G_g$ is the whole $S_n$. Our aim in this paper is to understand when $G_f$ is not the whole $G_0$.

  Group-theoretically, this problem behaves just as in \cite{ABORecipPolys}:
  
  \begin{prop}[\cite{ABORecipPolys}, Theorem 3.4] \label{prop:G_i}
    The maximal subgroups of $G_0$ whose projection onto $S_n$ is the whole group are
    \begin{alignat*}{2}
      G_1 &= \left\{(\mathbf{v}, \sigma) \in \FF_2^n \rtimes S_n : \sum_i v_i = 0\right\} = \<\1\>^\perp \rtimes S_n, &\quad n &\geq 1 \\
      G_2 &= \left\{(\mathbf{v}, \sigma) \in \FF_2^n \rtimes S_n : \sum_i v_i = \sgn \sigma\right\}, &\quad n &\geq 2 \\
      G_3 &= \<\1\> \cross S_n, &\quad n &\geq 3 \text{ odd}
    \end{alignat*}
  \end{prop}

  To prove Theorems \ref{thm:main} and \ref{thm:monic}, we must bound the number of polynomials $f$, equivalently $g$, for which $G_f$ is a subgroup of $G_1$, $G_2$, or $G_3$ for the values of $n$ listed in Proposition \ref{prop:G_i}.
  
  \begin{defn} \label{def:E_n}
    For $G \subseteq G_0$ and $H \geq 2$, let $\E_n(G; H)$ be the number of separable reciprocal polynomials $f$ of degree $2n$ with coefficients in $[-H, H]$ such that $G_g = S_n$ and $G_f \subseteq G$. Let $\E_n^{\monic}(G; H)$ count the subset of these that are monic.
  \end{defn}

  Applying the method of proof of Lemma 3.8 of \cite{ABORecipPolys} to Lemma \ref{lem:disc_f}, we get the following criteria:
  \begin{lemma}\label{lem:G123_conds} Assume that $G_g$ is the whole of $S_n$. Then:
    \begin{enumerate}[$($a$)$]
      \item\label{it:G1} $G_f \subseteq G_1$ if and only if $g\bigl(2\sqrt{b}\bigr) g\bigl(-2\sqrt{b}\bigr)$ is a square.
      \item\label{it:G2} $G_f \subseteq G_2$ if and only if $g\bigl(2\sqrt{b}\bigr) g\bigl(-2\sqrt{b}\bigr) \disc g$ is a square.
    \end{enumerate}
  \end{lemma}

\subsection{Counting \texorpdfstring{$G_1$}{G1}-polynomials}
From $G_1$ we get the main terms of Theorems \ref{thm:main} and \ref{thm:monic}.
\begin{theorem} \label{thm:G1}
  $\E_{n,b}(G_1; H) \asymp H^{n} \log H$, and $\E_{n,b}^\monic(G_1; H) \asymp H^{n - 1} \log H$.
\end{theorem}
\begin{proof}
For $g(u) = \sum_{i=0}^n c_i u^i$, write
\[
  g\bigl(\pm 2\sqrt{b}\bigr) = X \pm 2Y\sqrt{b},
\]
where the integer and radical parts
\[
  X = \sum_{i = 0}^{\floor{n/2}} c_{2i} (4b)^{i} \textand Y = \sum_{i = 0}^{\floor{(n-1)/2}} c_{2i + 1} (4b)^i
\]
depend only on the even and odd degree coefficients of $g$ respectively. Then $G_f \subseteq G_1$ if and only if $X$ and $Y$ satisfy the equation
\begin{equation} \label{eq:conic}
  \C : X^2 - 4b Y^2 = Z^2,
\end{equation}
which is a conic in the projective plane. By stereographic projection from the solution $(1,0,1)$, we parametrize all solutions up to scaling by
\begin{equation} \label{eq:param_conic}
  [X : Y : Z] = [s^2 + 4b t^2 : 2 s t : s^2 - 4b t^2],
\end{equation}
where $s$ and $t$ are coprime integers, with each pair $[s : t]$, up to negation, yielding a unique point on the conic $\C$. The coordinates $X$, $Y$, $Z$ as computed by \eqref{eq:param_conic} will be coprime after canceling out the common factor
\begin{align*}
  g &= \gcd(s^2 + 4b t^2, 2 s t, s^2 - 4b t^2) \\
  &\leq \gcd(2s^2, 8 b t^2) \\
  &\leq 8 b \gcd(s^2, t^2) \\
  &= 8b \\
  &\ll 1.
\end{align*}
Now a general solution
\[
  (X,Y,Z) = \left(\frac{k(s^2 + 4bt^2)}{g}, \frac{2k st}{g}, \frac{k(s^2 - 4bt^2)}{g}\right)
\]
has height
\begin{align*}
  \Ht(X,Y) &= \max\{|X|,|Y|\} \\
  &= \frac{k}{g} \max\{s^2 + 4b t^2, 2 s t\} \\
  &\asymp k \cdot \Ht(s,t)^2.
\end{align*}
Hence the number $N$ of solutions $(X,Y)$ of height at most $H$ is
\begin{align*}
  N &\asymp \sum_{k \ll H} \sum_{\substack{|s|,|t| \ll \sqrt{H/k} \\ \gcd(s,t) = 1}} 1 \\
  &\asymp \sum_{k \ll H} \frac{H}{k} \\
  &\asymp H \log H.
\end{align*}
Each pair $(X,Y)$ corresponds to $O(H^{n-1})$ polynomials $g$, establishing the desired upper bound $\E_{n,b}(G_1, H) \ll H^n \log H$. The lower bound is established similarly, noting that if $\Ht(X,Y) \leq cH$ for an appropriate $c$, there are in fact $\Theta(H^{n-1})$ polynomials $g$ for each $(X,Y)$.

The monic case proceeds identically. In either case, it is only necessary to have at least one free coefficient $c_i$ in both $X$ and $Y$, entailing $n \geq 1$ in general and $n \geq 2$ in the monic case. (In the monic case where $n = 1$, it is easy to see that only finitely many polynomials $f(x) = x^2 + a_1 x + b$ split into linear factors, which is the only way for $G_f$ to fall short of the full $G_0 = S_2$.)
\end{proof}

\subsection{Counting \texorpdfstring{$G_2$}{G2}-polynomials}\label{sec:G2}
In \cite{ABORecipPolys}, the most difficult group to deal with was $G_2$. We find here that the method works without essential change, so we briefly explain the adaptations needed to prove the following:
\begin{theorem} \label{thm:G2}
  For $n \geq 2$,
  \begin{align}
    \E_n(G_2; H) &\ll H^{n}  \label{eq:G2_nonmonic} \\
    \E_n^\monic(G_2; H) &\ll H^{n - 1}. \label{eq:G2_monic}
  \end{align}
\end{theorem}

By Lemma \ref{lem:G123_conds}\ref{it:G2}, we wish to count $g$ such that $g\bigl(2\sqrt{b}\bigr) g\bigl(-2\sqrt{b}\bigr) \disc g$ is a square. Let $\delta$ be a small constant, such as $1/(4n)$. We first prove the following analogue of Lemma 5.7 of \cite{ABORecipPolys}:
\begin{lemma} \label{lem:case1}
  Let $D$ be a positive integer. Let $C = \prod_{p\mid D} p$ be its radical; let $D'^2 = \prod_{p\mid D} p^{2\ceil{v_p(D)/2}}$ be its smallest square multiple. Assume that $C < H^{1 + \delta}$. The number of $b$-reciprocal integer polynomials $f$ of height $\leq H$ for which $G_f \subseteq G_2$ and $D \mid \Disc K_g$ is
  \[
  \ll \frac{O(1)^{\omega(C)} H^{n+1}}{D'^2}.
  \]
\end{lemma}
\begin{proof}
First we divide out from $D$ all primes $p \leq n$ and all primes $p \mid 2b$. If there is at least one $K_g$ with $D \mid \disc K_g$, this change only affects $D$ by a bounded factor, since $v_p(\disc K_g)$ is uniformly bounded.

For each prime $p \mid C$, excluding the degenerate case that $g \equiv 0$ mod $p$ (which is dealt with in the same way as in \cite{ABORecipPolys}), let $\sigma_p = (f_1^{e_1} \cdots f_r^{e_r})$ be the splitting type of the homogenization $\tilde g = y^n g(x/y)$. By \cite[Lemma 5.3, first inequality]{ABORecipPolys},
\[
  v_p(D) \leq \ind(\tilde g \bmod p).
\]
Also, if $v_p(D)$ is odd, the relations
\[
  1 \equiv v_p(D) \equiv v_p(\disc g) \equiv v_p\Big(g\bigl(2\sqrt{b}\bigr) g\bigl(-2\sqrt{b}\bigr)\Big) \mod 2
\]
imply, firstly, that $b$ is a square modulo $p$, for if $p$ is inert in the quadratic field $K_2 = \QQ\big(\sqrt{b}\big)$, then the norm
\[
  g\bigl(2\sqrt{b}\bigr) g\bigl(-2\sqrt{b}\bigr) = N_{K_2/\QQ} \bigl(2\sqrt{b}\bigr)
\]
must have even $p$-adic valuation; and secondly that $g(u)$ mod $p$ is divisible by $u - 2 \sqrt{b}$ for some choice of square root $\sqrt{b} \in \FF_p$. Then in defining the analogue of an annotated splitting type $(\sigma, j)$, the annotation $j$ should not be merely an integer $j \in \{0,2,-2\}$, but either $0$ or a choice of square root $j = 2 \sqrt{b} \in \FF_p$. In the latter case we set
\[
  w_{p,\sigma,j}(h) = w'_{p,\sigma}\big(x^n h(y/x - j)\big)
\]
and continue the argument of \cite{ABORecipPolys}. When $|j|/2$ appears as an index it must be replaced by
\[
  \begin{cases}
    0, & j = 0 \\
    1, & j = \pm 2\sqrt{b}.
  \end{cases}
\]
The proof then proceeds exactly as in \cite[\textsection 5.3]{ABORecipPolys}.
\end{proof}
As explained in the discussion following the statement of Lemma 5.7 in \cite{ABORecipPolys}, this lemma immediately establishes Theorem \ref{thm:G2} in Case I (following Bhargava's numbering of the cases), where $D \coloneqq \size{\Disc K_g} \geq H^{2 + 2\delta}$ and $C \coloneqq \prod_{p \mid D} p \leq H^{1 + \delta}$. The proof of Case II is unchanged from \cite[\textsection 5.4]{ABORecipPolys}, which is in turn unchanged from \cite[\textsection 5]{Bhargava_vdW}, since the (twisted) reciprocality of $f$ is not needed. Finally, for Case III, we construct a suitable analogue of the double discriminant by setting
\[
  h(c_0,\ldots, c_n) = g\bigl(2\sqrt{b}\bigr) g\bigl(-2\sqrt{b}\bigr) \disc g
\]
and
\begin{align*}
  R(c_0,\ldots, c_n) &= \Res_{c_0} \left(h, \frac{\partial}{\partial c_0} h\right) \\
  &= (-1)^{n(n-1)/2} c_n \disc_{c_0} h \\
  &= (-1)^{n(n-1)/2} c_n \disc_{c_0} \disc_u g \cdot \Bigl(g\bigl(2\sqrt{b}\bigr) - g\bigl(-2\sqrt{b}\bigr)\Bigr)^2 \cross \\
  &\quad {}\cross \Bigl(\disc_u \bigl(g - g\bigl(2\sqrt{b}\bigr)\bigr)\Bigr)^2 \Bigl(\disc_u \bigl(g - g\bigl(-2\sqrt{b}\bigr)\bigr)\Bigr)^2 
\end{align*}
(compare \cite[(13)]{ABORecipPolys}). Just as in \cite{ABORecipPolys}, we obtain that if $p \mid C$, then $p^2 \mid h$ for mod $p$ reasons, which implies $p \mid R$, and the rest of the proof proceeds as in \cite{ABORecipPolys}.

\subsection{Counting \texorpdfstring{$G_3$}{G3}-polynomials}\label{sec:G3}

Finally, we establish the following analogue of Theorem 6.1 of \cite{ABORecipPolys}.

\begin{theorem} \label{thm:G3}
  For $n \geq 3$ odd,
  \begin{align}
    \E_{n,b}(G_3; H) &\ll \begin{cases}
      H^2 \log^2 H & n = 3 \\
      H^{\frac{n+1}{2}} & n \geq 5
    \end{cases}  \label{eq:G3_nonmonic} \\
    \E_{n,b}^\monic(G_3; H) &\ll \begin{cases}
      H^2 & n = 3 \\
      H^2 \log H \log \log H & n = 5 \\
      H^{\frac{n-1}{2}} \log H & n \geq 7.
    \end{cases}
    \label{eq:G3_monic}
  \end{align}
\end{theorem}
\begin{proof}
Recall that the \emph{content} of a polynomial $f$, $\ct f$, is the $\gcd$ of its coefficients.  First, it suffices to count $f$ that are primitive (i.e.\ of content $1$), since if we can establish $\E_{n,b}^{\prim}(G_3; H) \ll H^{\frac{n+1}{2}}$ for $n \geq 5$, we get
\[
  \E_{n,b}(G_3; H) = \sum_{k \geq 1} \E_{n,b}^\prim\Big(G_3; \frac{H}{k}\Big) \ll H^{\frac{n+1}{2}} \sum_{k \geq 1} \frac{1}{k^{\frac{n+1}{2}}} \ll H^{\frac{n+1}{2}},
\]
and likewise for all the other cases.

The Galois condition $G_f \subseteq G_3$ occurs when there is a factorization
\begin{equation} \label{eq:f_fax}
  f(x) = c \cdot h(x) \cdot x^n h\left(\frac{b}{x}\right)
\end{equation}
defined either over $\QQ$ or over a quadratic field $K_2 = \QQ\bigl(\sqrt{k}\bigr)$. In the former case, we may scale $h$ such that $\ct h = 1$. Then $|c^{-1}| = \ct(x^n h(b/x))$ is a divisor of $b^n$ and in particular is bounded. So $\Ht h = \Htp h \ll \sqrt{H}$, and we get at most $H^{(n+1)/2}$ polynomials $f$, as in the first part of \cite[\textsection 6.2]{ABORecipPolys}.

We are left with the case that \eqref{eq:f_fax} holds with two degree $n$ factors defined over $K_2 = \QQ\bigl(\sqrt{k}\bigr)$ for some squarefree integer $k$. Let $\aa$ be the content ideal of $h$ (that is, the $\OO_{K_2}$-module generated by the coefficients of $h$). By scaling, we may assume that $\aa$ is integral. We have
\[
  \ct\biggl( x^n h\Bigl(\frac{b}{x}\Bigr) \biggr) = \aa \bb,
\]
where $\bb$ is a divisor of $b^n$. Taking contents of both sides of \eqref{eq:f_fax}, we get
\[
  (1) = (\ct f) = c \cdot \aa^2 \bb.
\]
There are $O(1)$ possibilities for $\bb$, and for each $\bb$, the class of $\aa$ is unique up to multiplying by one of the $\size{\Cl(K_2)[2]} \asymp 2^{\omega(k)}$ ideal classes of order dividing $2$. Let us fix $\aa$ and $\bb$. 

Since $f \in \ZZ[x]$, there must be a constant $c' \in K_2^\cross$ such that
\begin{equation} \label{eq:h_conj}
  x^n h\left(\frac{b}{x}\right) = c' \bar{h}(x).
\end{equation}
Comparing contents of \eqref{eq:h_conj}, we get
\[
  (c') = \frac{\aa \bb}{\bar \aa},
\]
so in particular $N(c') = N(\aa)N(\bb)/N(\bar \aa) = N(\bb)$ is a divisor of $b^n$. If $K_2$ is imaginary, we get $\size{c'}_\infty \asymp 1$. If $K_2$ has two real places $v, \bar v$, 
note that the scaling $h \mapsto \eta h$ causes $c' \mapsto \pm \eta^2 c'$. By scaling by an appropriate unit $\eta$, we can assume that $1 \leq \size{c'}_v < \size{\eta}_v^2$. Then there are $O(1)$ possible values of $c'$ given $\bb$ and $\aa$.

We have
\begin{align}
  \sqrt{H} &\geq \sqrt{\Ht f} \nonumber \\
  &= \sqrt{\Htp f} \nonumber \\
  &\asymp \Htp h \nonumber \\
  &= \prod_{v} \max\bigl\{\size{\theta_0}_v, \ldots, \size{\theta_n}_v\bigr\}^{[(K_2)_v : \QQ_v] / [K_2 : \QQ]} \nonumber \\
  &= \prod_{v \nmid \infty} N(\aa\OO_v)^{[(K_2)_v:\QQ_v]/2} \prod_{v\mid \infty} \max\bigl\{\size{\theta_0}_v, \ldots, \size{\theta_n}_v\bigr\}^{[(K_2)_v : \QQ_v] / 2} \nonumber \\
  &= N(\aa)^{-1/2} \prod_{v\mid \infty} \max\bigl\{\size{\theta_0}_v, \ldots, \size{\theta_n}_v\bigr\}^{[(K_2)_v : \QQ_v] / 2}. \label{eq:ht_bd}
\end{align}

In the case that $K_2$ is complex, this constrains the $\theta_i$ to lie in a disk $\D$ of area $O\big(N(\aa) H\big)$. Since $\aa$ has covolume $\Theta\big(\size{k} N(\aa)\big)$, the number of lattice points of $\aa$ in $\D$ is $O(H/\sqrt{\size{k}})$, using that there are at least three noncollinear points in $\D \intsec \aa$ since $h$ is not a scalar multiple of an integer polynomial. Also, by \eqref{eq:h_conj}, the first $(n+1)/2$ coefficients $\theta_0,\ldots, \theta_{(n+1)/2}$ determine the others, so we get $O(H^{(n+1)/2} \size{k}^{-(n+1)/4})$ polynomials $h$ for given $\bb$ and $\aa$.

In the case that $K_2$ is real, the two factors in \eqref{eq:ht_bd} are linked by \eqref{eq:h_conj}, which tells us that
\[
  \max_i \size{\theta_i}_v \asymp \size{c'}_v \max_i \size{\theta_i}_{\bar v} .
\]
Hence
\[
  \size{\theta_i}_{v} \ll \sqrt{\size{c'}_v N(\aa) H} \textand \size{\theta_i}_{\bar v} \ll \sqrt{\frac{N(\aa) H}{\size{c'}_v}}.
\]
So the $\theta_i$ (in the Minkowski embedding) are constrained to lie in a rectangle $\D$ of area $O\big(N(\aa) H\big)$. As in the preceding case, this yields $O(H^{(n+1)/2} k^{-(n+1)/4})$ polynomials $h$ for given $\bb$ and $\aa$. In particular, we must have $k \ll H^2$ to get any $h$ at all. Lastly, we sum over $\bb$, $\aa$ and $k$ to get
\begin{align}
  \E_{n,b}(G_3; H) &\ll H^{(n+1)/2}\sum_{\substack{k \ll H^2\\\text{squarefree}}} \frac{2^{\omega(k)}}{k^{(n+1)/4}} \nonumber \\
  &\leq H^{(n+1)/2} \sum_{\substack{a,b \ll H^2}} \frac{1}{a^{(n+1)/4} b^{(n+1)/4}} \nonumber \\
  &\ll \begin{cases}
    H^{(n+1)/2} \log^2 H & n = 3 \\
    H^{(n+1)/2} & n \geq 5.
  \end{cases}
\end{align}
The monic case is similar and simpler, as we can take $\aa = (1)$, so $(c') = \bb$ and $\theta_n$ is a unit. There are $O(\log H/\log \size{k})$ units within the appropriate height bounds, yielding
\begin{equation*}
  \E_n^{\monic}(G_3; H) \ll \sum_{2 \leq k \ll H^2} \left(\frac{H}{\sqrt{k}}\right)^{\frac{n-1}{2}} \cdot \frac{\log H}{\log k}
\end{equation*}
which gives the claimed numerics as in \cite{ABORecipPolys}.
\end{proof}

Note that the implicit constant depends on $b$ polynomially, as can be seen by examining our proof technique.  More carefully tracking this constant may be of interest, but we do not pursue this here.

\section{Other special polynomial families} \label{sec:fixed_con}
There are many other special polynomial families that we could consider.  That is, restricting to polynomials satisfying certain conditions, what is the Galois behavior of a randomly selected polynomial?  While surveying statistics of other interesting polynomial families is beyond the scope of this paper, we briefly outline partial progress for one such family in order to encourage future study along these lines.

The family we will highlight here are those with fixed constant term, which are also studied in \cite{Chavdarov}.  Consider
\begin{equation} \label{eq:fixed_const_term}
  f(x) = x^n + a_{n-1} x^{n-1} + \cdots + a_1 x + b
\end{equation}
where the outer coefficients $a_n = 1$, $a_0 = b$ are fixed. There are $O(H^{n-1})$ polynomials of this shape of height up to $H$, so the analogue of van der Waerden's conjecture in this setting is the following:
\begin{conj} \label{conj:fixed_con}
  For fixed $n \geq 2$ and $b \neq 0$, the number $E_{n,b}(H)$ of polynomials of the shape \eqref{eq:fixed_const_term} whose Galois group is not the full $S_n$ is bounded by
  \[
    E_{n,b}(H) \asymp H^{n-2}.
  \]
\end{conj}
The lower bound is evident by counting $f$ such that $f(1) = 0$. We mention some progress toward the upper bound, Proposition \ref{vdW:fixed_con} below, in hopes of encouraging further study.

\begin{rem}
This is a natural codimension $2$ sublattice of the family of all polynomials. Note that if we were to fix the constant term but \emph{not} the leading term, we could just as well fix only the leading term $a_n = b$, and then a transformation $f(x) \mapsto b^{n-1} f(x/b)$ sends the family of such $f$ into a finite-index sublattice of all monic polynomials. By the same transformation, a general pair of specifications $a_n = b, a_0 = c$ is reduced to the present case of monic with fixed constant term.
\end{rem}

To study the Galois group, we follow the general approach of Bhargava.  We start with reducible polynomials.
\subsection{Reducible polynomials}
We begin with the reducible case, that is, the case in which the Galois group is intransitive.
\begin{lemma}
  Let $R_{n,b}(H)$ count monic polynomials of constant term $b$ of height at most $H$. We have
  \[
    R_{n,b}(H) \asymp H^{n-2}.
  \]
\end{lemma}
\begin{proof}
 The lower bound may be proved by considering the family of polynomials
  \[
    f(x) = (x + 1)(x^{n-1} + d_{n-2} x^{n-2} + \cdots + d_{1}x + b)
  \]
  where $\size{d_i} \leq H/2$. For the upper bound, factor any reducible $f$ as
  \begin{align*}
    f(x) &= g(x)h(x), \\
    g(x) &= x^k + c_{k-1}x^{k-1} + \cdots + c_1 x + c_0, \\
    h(x) &= x^{n-k} + d_{n-k-1}x^{n-k-1} + \cdots + d_1 x + d_0,
  \end{align*}
  with coefficients $c_i, d_i \in \ZZ$ and $k \leq n/2$. The constant terms $c_0$ and $d_0$ have a finite number of possibilities: the divisors of $b$, and the remaining coefficients can be controlled via the height bound
  \[
    \Ht(g) \Ht(h) \asymp \Ht(f) \ll H
  \]
  (see \cite[equation (19)]{ABORecipPolys}). If $k = 1$, then there are only finitely many linear polynomials $g$, and counting the $h$ with $\Ht h \ll H$ yields at most $O(H^{n-2})$ reducible polynomials. If $k > 1$, the number of $f$ factoring into degrees $k$ and $n-k$ is
  \begin{align*}
    R_{n,b,k}(H) &\leq \sum_{m \ll H} \size{\{g : \Ht g = m\}} \cdot \Size{\left\{h : \Ht h \ll \frac{H}{m}\right\}} \\
    &\ll \sum_{m \ll H} m^{k-2} \left(\frac{H}{m}\right)^{n-k-1} \\
    &= H^{n-k-1} \sum_{m\ll H} m^{n - 2k - 1} \\
    &\ll \begin{cases}
        H^{n-k-1}, & k < n/2 \\
        H^{n-k-1} \log H, & k = n/2.
    \end{cases}
  \end{align*}
  In all cases $R_{n,b,k}(H) \ll H^{n-2}$, so $R_{n,b}(H) \ll H^{n-2}$.
\end{proof}

\subsection{Imprimitive polynomials}
Next, the Galois group $G_f$ may be transitive but imprimitive, meaning that the associated number field $K_f = \QQ[x]/f(x)$ contains a subfield, $K_f \supsetneq L \supsetneq \QQ$. The maximal imprimitive subgroups of $S_n$ are the wreath products $S_a \wr S_b$, where $a b = n$; in particular, $n$ must be composite. An excellent bound for the number of polynomials generating an imprimitive extension is given by Widmer \cite[Theorem 1.1]{Widmer11}. It implies \cite[equation (5)]{Bhargava_vdW} that the number of imprimitive polynomials of height at most $H$ is at most $O(H^{n/2 + 2})$. If $n \geq 8$, then $n/2 + 2 \leq n - 2$, so Widmer's bound is adequate for our purposes. As imprimitive extensions must have composite order, we are left with the cases $n = 4$ and $n = 6$.  For these two remaining cases, the results of Chow and Dietmann on quartic and sextic polynomials via resolvents \cite{CD2020}, \cite{CD23_Towards} are not applicable out of the box, but a deeper dive into these results and techniques is likely to resolve these cases.

Finally, we can reduce to studying transitive, primitive polynomials.
\subsection{Primitive polynomials}
Bhargava divides the remaining polynomials into three cases (Cases I, II, and III in his notation).  For Cases I and III we can mimic his proof exactly.  For Case II, we can use Schmidt's bound (and improvements) when $\frac{n+2}{2}+1 \leq n-2$, i.e.\ $n \geq 8$. We can also use Bhargava's exact asymptotics in place of Schmidt's bound for $n = 5$ with an $\varepsilon$ loss). As with the imprimitive case, only small degrees $n$ (here $n \leq 7$) remain unverified, and we expect these degrees will be approachable via resolvents in the style of the Chow--Dietmann papers mentioned above.

Hence we have a van der Waerden result for $n$ large enough:
\begin{prop} \label{vdW:fixed_con}
  For fixed $n \geq 8$ and $b \neq 0$, the number $E_{n,b}(H)$ of polynomials of the shape \eqref{eq:fixed_const_term} whose Galois group is not the full $S_n$ is bounded by
  \[
    E_{n,b}(H) \asymp H^{n-2}.
  \]
\end{prop}
As Bhargava's techniques are quite flexible, it is likely that other linear families of polynomials can be understood.  We leave this for future study.

\bibliography{ourbib_3}
\bibliographystyle{alpha}
\end{document}